\documentclass[12pt]{amsart}

\usepackage{graphicx}
\usepackage{amscd}
\usepackage{amsmath}
\usepackage{amssymb}
\usepackage{amsthm}
\usepackage{enumerate} 
\usepackage{mathtools} 
\usepackage{amsfonts}
\usepackage[margin=1.2in]{geometry}

\theoremstyle{plain}
\newtheorem{theorem}{Theorem}[section]
\newtheorem{lemma}[theorem]{Lemma}

\theoremstyle{definition}
\newtheorem{definition}[theorem]{Definition}
\newtheorem{remark}[theorem]{Remark}

\numberwithin{equation}{section}

\newcommand\naturals{\mathbb{N}}
\newcommand\integers{\mathbb{Z}}

\newcommand\reals{\mathbb{R}}

\newcommand\norm[1]{\left\lVert#1\right\rVert}
\newcommand\ev[2]{\left<#1,#2\right>}

\DeclareMathOperator{\supp}{supp}

\newcommand\littleoh[1]{o\left(#1\right)}

\newcommand\ds{\mathcal{D}}

\newcommand\es{\mathcal{E}}

\newcommand\sss{\mathcal{S}}

\newcommand\beurling[2]{\left(#1_{#2}\right)}
\newcommand\beurlingM{\beurling{M}{p}}

\newcommand\roumieu[2]{\left\{#1_{#2}\right\}}
\newcommand\roumieuM{\roumieu{M}{p}}

\newcommand\beurou{\ast}
\newcommand\utestspace[3]{#1^{#2}\left(#3\right)}
\newcommand\utestspacereals[3]{\utestspace{#1}{#2}{\reals^{#3}}}

\newcommand\utestspacereal[2]{#1^{#2}}
\newcommand\udspace[2]{\utestspace{\ds}{#1}{#2}}
\newcommand\udspacereals[2]{\utestspacereals{\ds}{#1}{#2}}
\newcommand\udspacereal[1]{\utestspacereal{\ds}{#1}}

\newcommand\uespacereals[2]{\utestspacereals{\es}{#1}{#2}}
\newcommand\uespacereal[1]{\utestspacereal{\es}{#1}}

\newcommand\uddspace[2]{\utestspace{\ds}{#1 \prime}{#2}}

\newcommand\uddspacereal[1]{\utestspacereal{\ds}{#1 \prime}}

\newcommand\udespacereal[1]{\utestspacereal{\es}{#1 \prime}}

\begin{document}

\title[Quasiasymptotics of ultradistributions]{Structural theorems for quasiasymptotics of ultradistributions}

\author[L. Neyt]{Lenny Neyt}

\author[J. Vindas]{Jasson Vindas}

\address{Department of Mathematics: Analysis, Logic and Discrete Mathematics\\ Ghent University\\ Krijgslaan 281\\ 9000 Gent\\ Belgium}
\email{lenny.neyt@UGent.be, jasson.vindas@UGent.be}

\thanks{The authors  gratefully acknowledge support by Ghent University, through the BOF-grant 01J11615.}



\begin{abstract}
We provide complete structural theorems for the so-called quasiasymptotic behavior of non-quasianalytic ultradistributions.  As an application of these results, we obtain descriptions of quasiasymptotic properties of regularizations at the origin of ultradistributions and discuss connections with Gelfand-Shilov type spaces.

\end{abstract}

\maketitle

\section{Introduction}

Several asymptotic notions play a fundamental role in the theory of generalized functions. The subject has been studied by several authors and applications have been elaborated in areas such as mathematical physics, Tauberian theorems for integral transforms, number theory, and differential equations. See the monographs \cite{estrada2012distributional, p-s-t, stevan2011asymptotic, vladimirov1988tauberian} for an overview of results and the articles \cite{drozhzhinov2016,P-V2014,yang-estrada} for recent contributions. 

The purpose of this article  is to present a detailed structural study of the so-called quasiasymptotics of ultradistributions. The concept of quasiasymptotic behavior for Schwartz distributions was introduced by Zav'yalov in \cite{Zavyalov} and further developed by him, Drozhzhinov, and Vladimirov in connection with their powerful multidimensional Tauberian theory for Laplace transforms \cite{vladimirov1988tauberian}. A key aspect in the understanding of this concept is its description via so-called structural theorems and complete results in that direction were achieved in \cite{StructThQADistrInf,StructThQADistrOrigin} (cf. \cite{lojasiewicz57,stevan2011asymptotic}). In \cite{10.2307/44095755} Pilipovi{\'c} and Stankovi{\'c} naturally extended the definition of quasiasymptotic behavior to the context of one-dimensional ultradistributions and studied its basic properties. We shall obtain here complete structural theorems for quasiasymptotics of non-quasianalytic ultradistributions that generalize their distributional counterparts. Our main goal is thus to characterize those ultradistributions having quasiasymptotic behavior as infinite sums of derivatives of functions satisfying classical pointwise asymptotic relations. 

The paper is organized as follows. In Section \ref{preliminaries section} we explain some notions and tools that will play a role in our arguments. Section \ref{Section structure at infinity} studies the quasiasymptotic behavior at infinity. A key idea we apply here will be to connect the quasiasymptotic behavior with the so-called $S$-asymptotic behavior \cite{stevan2011asymptotic}, for which structural theorems are available, via an exponential change of variables. The nature of the problem under consideration requires to split our treatment in two cases, depending on whether the degree of the quasiasymptotic behavior is a negative integer or not.  We obtain in Section \ref{structure quasiasymptotic at the origin} structural theorems for the quasiasymptotic behavior at the origin. Our technique there is based on a reduction to the results from Section \ref{Section structure at infinity} by means of a change of variables and then regularization. Our method also yields asymptotic properties of regularizations at the origin of ultradistributions having prescribed asymptotic properties, generalizing results for distributions from \cite{VindasReg2011}. It is also worth mentioning that our approach here differs from the one employed in the literature to deal with Schwartz distributions, and in fact can be used to produce new proofs for the classical structural theorems for the quasiasymptotic behavior of distributions. We conclude the article by studying extensions of quasiasymptotics to new ultradistributions spaces of Gelfand-Shilov type that we shall introduce in Section \ref{quasiasymptotics Gelfand-Shilov}.

\section{Preliminaries}
\label{preliminaries section}

Throughout this article we fix
a weight sequence of positive  numbers $\{ M_{p} \}_{p \in \naturals}$
and assume it satisfies $(M.1)$, $(M.2)$, and $(M.3)$;
the meaning of all these conditions is very well explained in \cite{ultradistributions1}. Let $\Omega\subseteq\mathbb{R}$.
As customary \cite{ultradistributions1}, $\mathcal{E}^{\ast}(\Omega)$ and $\mathcal{D}^{\ast}(\Omega)$ stand, respectively, for the spaces of all ultradifferentiable functions and compactly supported ultradifferentiable functions of class $\ast$ on $\Omega$, where we employ $\ast$ as
 the common notation for $\beurlingM$ and $\roumieuM$. For statements needing a separate treatment we will always talk first about the Beurling case, followed by the corresponding assertion for the Roumieu case in parenthesis. When $\Omega = \reals$, we simply write $\udspacereal{\beurou} = \udspacereals{\beurou}{}$ and $\uespacereal{\beurou} = \uespacereals{\beurou}{}$. The strong duals $\uddspacereal{\beurou}(\Omega)$ and $\udespacereal{\beurou}(\Omega)$ are the spaces of ultradistributions and compactly supported ultradistributions, respectively, on $\Omega$. 

The main subject of study of this article is the quasiasymptotic behavior of ultradistributions, which is defined via asymptotic comparison with regularly varying functions. A real-valued measurable function $L$ is called \textit{slowly varying at infinity} \cite{bingham1989regular} if $L$ is positive for large arguments and
$L(ax)\sim L(x)$ as $x\to\infty$,
for any $a > 0$. 
We are only interested in the terminal behavior of $L$, so \cite{bingham1989regular} we may always assume $L$ to be defined, positive, and locally bounded (or even continuous) on $[0, \infty)$. Finally, we say that a function $L$ on $(0, \infty)$ is \emph{slowly varying at the origin} if $\widetilde{L}(x):=L(x^{-1})$ is slowly varying at infinity.

In accordance to \cite{10.2307/44095755,stevan2011asymptotic}, we define the quasiasymptotic behavior of an ultradistribution at infinity or at the origin as follows.

	\begin{definition}
		Let $L$ be a slowly varying function at infinity (at the origin, resp.). We say that $f \in \uddspacereal{\beurou}$ has \textit{quasiasymptotic behavior at infinity (at the origin) in $\uddspacereal{\beurou}$ with respect to $\lambda^{\alpha} L(\lambda)$}, $\alpha \in \reals$, if for some $g \in \uddspacereal{\beurou}$ and every $\phi \in \udspacereal{\beurou}$,
			\begin{equation}
				\label{eq:quasiasympinfinity}
				\lim_{\lambda \rightarrow \infty} \ev{\frac{f(\lambda x)}{\lambda^{\alpha} L(\lambda)}}{\phi(x)} = \ev{g(x)}{\phi(x)} \qquad \left(\mbox{resp., } \lim_{\lambda \rightarrow 0^{+}}\right) .
			\end{equation}
		If (\ref{eq:quasiasympinfinity}) holds, we also say that $f$ has \textit{quasiasymptotics of order $\alpha$ at infinity (at the origin) with respect to $L$} and write in short: $f(\lambda x) \sim \lambda^{\alpha} L(\lambda) g(x)$ in $\uddspacereal{\beurou}$ as $\lambda \rightarrow \infty$ (resp., $\lambda\to0^{+}$).
	\end{definition}

If $f(\lambda x) \sim \lambda^{\alpha} L(\lambda) g(x)$ in $\uddspacereal{\beurou}$ as $\lambda \rightarrow \infty$ (as $\lambda \rightarrow 0^{+}$), it can easily be shown \cite{estrada2012distributional,stevan2011asymptotic} that this forces $g$ to be a homogeneous ultradistribution of degree $\alpha$. An adaptation of the proof of \cite[Theorem 2.6.1]{estrada2012distributional} shows that all homogeneous ultradistributions are exactly the homogeneous distributions, which we denote exactly as in \cite{estrada2012distributional}. 
We mention that we will employ the notation $H(x)=x^{0}_{+}$ for the Heaviside function. 
In addition, we shall make use of the special (non homogeneous!) distributions $\operatorname*{Pf}(H(\pm) x^{-k})$, $k\in\mathbb{Z}_{+}$, where $\operatorname*{Pf}$ stands for Hadamard finite part regularization \cite{estrada2012distributional}.

\begin{remark}	Naturally \cite{stevan2011asymptotic}, the quasiasymptotic behavior may be defined in other spaces of generalized functions $\mathcal{F}'$ by asking (\ref{eq:quasiasympinfinity}) to hold for any $\phi \in \mathcal{F}$, whenever the dialation operators act continuously on the test function space.
\end{remark}

\section{The structure of quasiasymptotics at infinity}
\label{Section structure at infinity}
	
This section is devoted to studying the quasiasymptotic behavior at infinity. Our main results are Theorem \ref{t:structalpha>-1} and Theorem \ref{t:fstructnid}, where we provide a full description of the structure of quasiasymptotics at infinity. Some auxiliary lemmas used in their proofs are shown in Subsection \ref{Subsection some lemmas}. Throughout this section $L$ stands for a slowly varying function at infinity.

\subsection{Some lemmas}
\label{Subsection some lemmas}
We start with the ensuing useful estimates for the weight sequence $M_p$, which we shall often exploit throughout the article. Hereafter $S(n,k)$ stand for the \textit{Stirling numbers of the second kind} (see e.g. \cite{Johnson02thecurious}).

	\begin{lemma}
		\label{l:Msumfacgrowth}
		 For any $\ell > 0$ there is $C_{\ell} > 0$ (independent of $p$) such that
			\begin{equation}
				\label{eq:Msumfacgrowth}
				\sum_{k = p}^{\infty} \frac{k! \ell^{k}}{M_{k}} \leq C_{\ell} \frac{p!}{M_{p}} \ell^{p}  
			\end{equation}
and
\begin{equation}
				\label{eq:Msumgrowth}
				\sum_{k = p}^{\infty} S(k + 1, p + 1) \frac{\ell^{k}}{M_{k}} \leq C_{\ell} \frac{(2\ell)^{p}}{M_{p}} .
			\end{equation}
	\end{lemma}
	
	\begin{proof} 
	Clearly, it is enough to show \eqref{eq:Msumfacgrowth} just for sufficiently large $p$. Using \cite[Lemma 4.1, p. 55]{ultradistributions1}, there is $p_0$ such that for any $p\geq p_0$ we have $p / m_{p}:=p M_{p-1}/M_{p} \leq (2 \ell)^{-1}$. Hence, it follows that for $p$ in this range
			\begin{align*}
				\sum_{k = p}^{\infty} \frac{k! \ell^{k}}{M_{k}} &= \frac{p!}{M_{p}} \left( \ell^{p} + \sum_{k = p + 1}^{\infty} \frac{(p + 1) \cdot \ldots \cdot k \cdot \ell^{k}}{m_{p + 1} \cdot \ldots \cdot m_{k}} \right)
			\leq 2 \frac{p!}{M_{p}} \ell^{p} . 
			\end{align*}				
For \eqref{eq:Msumgrowth}, 
in view of \cite[Theorem 3]{StirlingNumbersOfSecondKind}, we have
		$S(k + 1, p + 1) \leq 2^{k + 1} (p + 1)^{k - p} \leq 2^{k + 1} k!/p! $
		for $k \geq p$. The rest follows by application of \eqref{eq:Msumfacgrowth}.
	\end{proof}

In \cite{StructThQADistrInf}, the structure of distributional quasiasymptotics at infinity was found by noting that certain primitives preserve the asymptotic behavior, being of a higher degree, and using the fact that eventually the primitives are continuous functions. As the latter part does not hold in general for ultradistributions, a more careful analysis is needed, although we may carry over some of the distributional results. In fact, one may retread the proofs from \cite[Section 2]{StructThQADistrInf} (see also \cite[Section 2.10]{stevan2011asymptotic}) 
to obtain,

	\begin{lemma}
		\label{l:qaprimitiveinfinity}
		Let $f \in \uddspacereal{\beurou}$. Suppose $f$ has quasiasymptotics with respect to $\lambda^{\alpha} L(\lambda)$.
			\begin{enumerate}[(i)]
				\item If $\alpha \notin \integers_{-}$: for any $n \in \naturals$ and any $n$-primitive $F_{n}$ of $f$ there exists a polynomial $P$ of degree at most $n - 1$ such that $F_{n} + P$ has quasiasymptotics with respect to $\lambda^{\alpha + n} L(\lambda)$ in $\uddspacereal{\beurou}$. 
				\item If $\alpha = -k$, $k \in \integers_{+}$: there is some $(k - 1)$-primitive $F$ of $f$ such that $F$ has quasiasymptotics with respect to $\lambda^{-1}L(\lambda)$ in $\uddspacereal{\beurou}$.
			\end{enumerate}
	\end{lemma}

The previous lemma roughly speaking shows that in order to find the structure of quasiasymptotics for arbitrary degree, it suffices to discover the structure for degrees $\geq -1$, where extra care is needed for the case $-1$. It should also be noticed that the converse statements for $(i)$ and $(ii)$ from Lemma \ref{l:qaprimitiveinfinity} trivially hold true.

The next lemma, a direct consequence of well-known moment asymptotic expansion \cite{estrada2012distributional,Schmidt2005},  states that the quasiasymptotic behavior of degree $>-1$ is a local property at infinity, which in some arguments enables us to remove the origin from the support of the ultradistribution in our analysis.

	\begin{lemma}
		\label{l:qalocalproperty}
		Suppose that $f_{1}, f_{2}  \in \uddspacereal{\beurou}$ and that for some $a > 0$, $f_{1}$ and $f_{2}$ coincide on $\reals \setminus [-a, a]$. Suppose that $f_{1}(\lambda x)\sim \lambda^{\alpha}L(\lambda) g(x)$ in $\uddspacereal{\beurou}$ as $\lambda\to\infty$, where $\alpha > -1$. Then, also $f_{2}(\lambda x)\sim \lambda^{\alpha}L(\lambda) g(x)$ in $\uddspacereal{\beurou}$.
	\end{lemma}

\subsection{Structural theorem for $\alpha \notin \integers_{-}$}

We study in this subsection quasiasymptotics of degree $\alpha\notin\mathbb{Z}_{-} $. Part of our analysis reduces the general case to that when $\alpha>-1$, i.e., the case when the quasiasymptotic behavior is local. Consequently, we may restrict our discussion to those ultradistributions whose support lie in the complement of some zero neighborhood. As both the negative and positive half-line can be treated symmetrically, it is natural to start the analysis with ultradistributions that are supported on the positive half-line. In the next crucial lemma we further normalize the situation by assuming that our ultradistribution is supported in $(e,\infty)$.

\begin{lemma}
\label{l:quasiStrucSup} Let $\alpha\in\mathbb{R}$ and let $f \in \uddspacereal{\beurou}$ be such that $\operatorname*{supp} f\subset (e,\infty)$ and $f$ has quasiasymptotic behavior at infinity with respect to $\lambda^{\alpha}L(\lambda)$ in $\uddspacereal{\beurou}(0,\infty)$. Then, there are continuous functions $f_{m}$  such that $\operatorname*{supp} f_{m}\subset (e,\infty)$,
		$
				f = \sum_{m = 0}^{\infty} f_{m}^{(m)} , 
		$
		the limits
			\begin{equation*}
				\lim_{x \rightarrow \infty} \frac{f_{m}(x)}{x^{\alpha + m} L(x)}
			\end{equation*}
		exist, and furthermore, for some $\ell > 0$ (any $\ell > 0$) there is a $C=C_{\ell} > 0$ such that,
			\begin{equation*}		
				\left| f_{m}(x) \right| \leq  C \frac{\ell^{m}}{M_{m}} x^{\alpha + m} L(x), \qquad m\in\mathbb{N},\ x>0. 			\end{equation*}

\end{lemma}
\begin{proof}
Suppose $f(\lambda x) \sim \lambda^{\alpha} L(\lambda) g(x)$ in $\uddspacereal{\beurou}(0, \infty)$ as $\lambda \rightarrow \infty$. Since composition with a real analytic function induces continuous mappings between spaces of ultradifferentiable functions (see e.g. \cite[Prop. 8.4.1, p. 281]{hormander1990analysis}), we obtain that the composition $f(e^{x})$ is an element of $\uddspacereal{\beurou}$. Also, $\psi \in \udspacereal{\beurou}$ if and only if  $\psi(x) = \varphi(e^{x})$ with $\varphi\in \udspacereal{\beurou}(0,\infty)$. 

These key observations allow us to make a change of variables in order to apply the structural theorem for $S$-asymptotics \cite[Theorem~1.10, p.~46]{stevan2011asymptotic}. 
In fact, we set $u(x) := f(e^{x})$, $w(x) := g(e^{x})$ and $c(h) := e^{\alpha h} L(e^{h})$. (Notice that $w$ has actually the form $w(x)=Be^{\alpha x}$ for some $B>0$.) A quick computation shows that that $u$ has the $S$-asymptotic behavior
$u(x + h) \sim c(h) w(x)$  in $\uddspacereal{\beurou} $ as $h \rightarrow \infty.$ The  quoted structure theorem yields the existence of an ultradifferential operator $P(D)$ of class $\beurou$ and continuous functions $u_1$ and $u_2$ with supports on $(1,\infty)$ such that $u = P(D) u_{1} + u_{2} $ 
on $(0, \infty)$ and $\lim_{h \rightarrow \infty}  u_{i}(x + h)/c(h)$ exist uniformly for $x$ on compacts of $(0,\infty)$.

Take any $\varphi \in \udspacereal{\beurou}(0, \infty)$ and put $\psi(x) = e^{x} \varphi(e^{x})$, then the substitution $y = e^{x}$ yields
	\begin{align*}
		\ev{f(y)}{\varphi(y)} &= \ev{f(y)}{\frac{\varphi(y)}{y} y} = \ev{u(x)}{\psi(x)} 
				= \ev{P(D) u_{1} (x)}{\psi(x)} + \ev{u_{2} (x)}{\psi(x)} . 
	\end{align*}
	
Let us consider both terms of the sum individually. The latter is simply	\[ \ev{u_{2}(x)}{\psi(x)} = \int_{1}^{\infty} u_{2}(x) \psi(x) dx  = \int_{e}^{\infty} u_{2}(\log y) \varphi(y) dy . \]
Setting $f_{2}(y) := u_{2}(\log y)$, we get
$		\ev{u_{2}(x)}{\psi(x)} = \ev{f_2(y)}{\varphi(y)} ,$
and the existence of 		
$\lim_{y \rightarrow \infty} y^{-\alpha}f_2(y)/ L(y) .
$ For the first term, we will need to explicitly calculate the derivatives of $\psi$. Using the Fa\`{a} di Bruno formula \cite[Eq. (2.2)]{Johnson02thecurious}, 

	\begin{align*}
		\psi^{(n)}(x) &= e^{x} \sum_{k = 0}^{n} {n \choose k} \frac{d^{k}}{dx^{k}} \left( \varphi(e^{x}) \right) 
				   = e^{x} \sum_{m = 0}^{n} S(n + 1, m + 1) e^{mx} \varphi^{(m)}(e^{x}) , 
	\end{align*}
where we have applied \cite[Theorem 5.3.B]{AdvancedCombinatorics}. 
Then,
	\begin{align*}
		\int_{1}^{\infty} u_{1}(x) \psi^{(n)}(x) dx
											 = \sum_{m = 0}^{n} S(n + 1, m + 1) \int_{e}^{\infty} u_{1}(\log y) \varphi^{(m)}(y) y^{m} dy . 
	\end{align*}
If $P(D) = \sum_{n = 0} a_{n} D^{n}$, then by \eqref{eq:Msumgrowth} from Lemma \ref{l:Msumfacgrowth}, we may consider the following constants,
	\begin{equation*}
		c_{m} = (-1)^{m} \sum_{n = m}^{\infty} (-1)^{n} a_{n} S(n + 1, m + 1) ,
	\end{equation*}
and it follows that $c_{m} \leq C \mu^{m} / M_{m}$ for some $\mu > 0$ (for any $\mu > 0$) and some $C_{\mu} = C > 0$. Collecting everything together, we obtain
	\begin{equation*}
		\ev{P(D) u_{1}(x)}{\psi(x)} = \sum_{m = 0}^{\infty} (-1)^{m} c_{m} \int_{e}^{\infty} u_{1}(\log y) \varphi^{(m)}(y) y^{m} dy . 
	\end{equation*}
So if we define $f_{1, m}(y) := u_{1}(\log y) y^{m}$, $m \in \naturals$, we get
	\begin{equation*}
		\ev{P(D) u_{1}(x)}{\psi(x)} = \sum_{m = 0}^{\infty} c_{m} \ev{f_{1, m}^{(m)}(y)}{\varphi(y)} ,
	\end{equation*}
and the limits
$		\lim_{y \rightarrow \infty} y^{-\alpha-m}f_{1, m}(y)/L(y)$ 
exist. This completes the proof of the lemma.

\end{proof}

We are ready to discuss the general case.

	\begin{theorem}
		\label{t:structalpha>-1}
		Suppose $\alpha \notin \integers_{-}$ and let $k \in \naturals$ be the smallest non-negative integer such that $-(k + 1) < \alpha$.  Then, an ultradistribution $f \in \uddspacereal{\beurou}$ has quasiasymptotic behavior 
\begin{equation}\label{eqquaalpha>-1}
f(\lambda x)\sim \lambda^{\alpha}L(\lambda)(c_{-}x_{-}^{\alpha}+c
_{+}x_{+}^{\alpha}) \quad \mbox{ in }\uddspacereal{\beurou} \mbox{ as }\lambda\to\infty
\end{equation}
 if and only if there exist continuous functions $f_{m}$ on $\reals$, $m \geq k$, such that
			\begin{equation}
				\label{eq:fstructalpha>-1}
				f = \sum_{m = k}^{\infty} f_{m}^{(m)} , 
			\end{equation}
the limits			
			\begin{equation}
				\label{eq:fstructalpha>-1asymp}
				\lim_{x \rightarrow \pm \infty} \frac{f_{m}(x)}{x^{m}|x|^{\alpha } L(|x|)}=c^{\pm}_{m}, \qquad m \geq k, 
			\end{equation}
		exist, and for some $\ell > 0$ (any $\ell > 0$) there is a $C=C_{\ell} > 0$ such that
			\begin{equation}
				\label{eq:fstructalpha>-1bound}
				\left| f_{m}(x) \right| \leq   C \frac{\ell^{m}}{M_{m}} (1+|x|)^{\alpha + m} L(|x|) , \qquad x\in\mathbb{R},
			\end{equation}
		for all $m \geq k$. Furthermore, in this case we have
                \begin{equation}
                \label{eq:strucalpha>-1constants}
                c_{\pm}= \sum_{m=k}^{\infty} c^{\pm}_{m} \frac{\Gamma(\alpha+m+1)}{\Gamma(\alpha+1)}.
                \end{equation}		
	\end{theorem}
	
	\begin{proof} In view of Lemma \ref{l:qaprimitiveinfinity}(i), we may assume that $\alpha>-1$ so that $k=0$. 
	
	Suppose then first that  $f $ has quasiasymptotic behavior \eqref{eqquaalpha>-1}. We write $f=f_{-}+f_{c}+f_{+}$, where $f_{c} \in \udespacereal{\beurou}$ coincides with $f$ on an open interval containing $[-e, e]$ and $\operatorname*{supp}f_{-}\subset (-\infty,-e)$ and $\operatorname*{supp}f_{+}\subset (e,\infty)$. Then, by Lemma \ref{l:qalocalproperty} each $f_{\pm}$ has quasiasymptotic behavior with respect to $\lambda^{\alpha} L(\lambda)$ in $\uddspacereal{\beurou}(-\infty, 0)$ and $\uddspacereal{\beurou}(0,\infty)$, respectively. Using Lemma \ref{l:quasiStrucSup}, we find continuous functions $f_{1, m}^{\pm}$, $m \in \naturals$ with supports in $(-\infty,-e)$ and $(e, \infty)$, respectively, such that  the identities
			$ f_{\pm} = \sum_{m = 0}^{\infty} ( f_{1, m}^{\pm} )^{(m)}$
	hold, the limits
$$
c^{\pm}_{m}=(-1)^{m}\lim_{x\to\infty} \frac{f_{1,m}^{\pm}(\pm x)}{x^{\alpha+m}L(x)}
$$	
exist, and the bounds $|f_{1,m}^{\pm}(x)|\leq C' \ell^{m} |x|^{\alpha + m} L(|x|)/ M_{m} $ are satisfied for some $\ell > 0$ (any $\ell > 0$) and some $C'=C'_{\ell} > 0$. Applying Komatsu's first structural theorem\footnote{The first structural theorem even holds true in the quasianalytic case under mild conditions, see \cite[Proposition 4.1 and Proposition 4.7]{debrouwere-vindasCousin}.} for ultradistributions  \cite[Theorem 8.1 and Theorem 8.7]{ultradistributions1} one can also find continuous functions $g_m$, whose supports lie in some (arbitrarily chosen) neighborhood of $\supp f_{c}$, such that $f_{c} = \sum_{m=0}^{\infty}g^{(m)}_{m}$ in $\uddspacereal{\beurou}$ and $\sup_{x\in\mathbb{R}}|g_{m}(x)|\leq C'' \ell^{m}/M_m$ for some $\ell>0$ (for every $\ell>0$) and $C''=C''_{\ell}>0$. The functions
			$ f_{m} = g_{m} + f_{1, m}^{-} +  f_{1, m}^{+} $
satisfy all sought requirements.	We verify the relation \eqref{eq:strucalpha>-1constants} below.
		
Conversely, assume that $f$ satisfies all of the conditions above. Take any $\phi \in \udspacereal{\beurou}$ and suppose that for some $R > 1$ we have $\supp \phi \subseteq [-R, R]$. Pick $\gamma > 0$ such that $\alpha - \gamma> -1$.
Using Potter's estimate  \cite[Theorem 1.5.4]{bingham1989regular}, we may assume that

	\begin{equation}
		\label{eq:slowlyvaryingsubpolynomialbound}
		\frac{L(\lambda x)}{L(\lambda)} \leq C_{\gamma} \max\{ x^{- \gamma}, x^{\gamma} \} 
	\end{equation}	
holds for all $x,\lambda>0$. Since $\phi \in \udspacereal{\beurou}$, for any $h > 0$ (for some $h > 0$) there exists a $C_{\phi,h}$ such that for all $m \in \naturals$ and $x \in \reals$ we have $|\phi^{(m)}(x)| \leq C_{\phi,h} h^{m} M_{m}$. Due to (\ref{eq:fstructalpha>-1bound}), we now have for any $m \in \naturals$ and $\lambda>1$
			\begin{align*}
				&\left| \frac{1}{\lambda} \int_{- \infty}^{\infty} \frac{f_{m}(x)}{\lambda^{\alpha} L(\lambda)} \frac{\phi^{(m)}(x / \lambda)}{\lambda^{m}} dx \right|
					 \\
					 &
					\leq C_{\ell} \frac{(2\ell)^{m}}{M_{m}} \left( \int_{-\lambda^{-1}}^{\lambda^{-1}} \frac{|\phi^{(m)}(x)|}{\lambda^{m + \alpha} L(\lambda)} dx + \int_{|x|\geq 1/ \lambda} \frac{L(\lambda |x|)}{L(\lambda)} |x|^{\alpha + m} |\phi^{(m)}(x)| dx \right) 
					\\
					&
					\leq 2 C_{\ell} C_{\phi,h} (2h \ell)^{m} \left( \frac{1}{\lambda^{m + 1 + \alpha} L(\lambda)} + C_{\gamma} R^{m + \alpha + \gamma + 1} + \frac{C_{\gamma}}{\alpha - \gamma + m + 1} \right) 
					\leq C (2h\ell R)^{m},  	
			\end{align*}
		and, as $2h \ell$ may be chosen freely, this is absolutely summable over $m \in \naturals$.
It follows by applying the Lebesgue dominated convergence theorem twice that
			\begin{align*}
				\lim_{\lambda \rightarrow \infty } \ev{\frac{f(\lambda x)}{\lambda^{\alpha} L(\lambda)}}{\phi(x)} 
					&= \lim_{\lambda \rightarrow \infty} \frac{1}{\lambda} \sum_{m = 0}^{\infty} (-1)^{m} \int_{- \infty}^{\infty} \frac{f_{m}(x)}{\lambda^{\alpha} L(\lambda)} \frac{\phi^{(m)}(x / \lambda)}{\lambda^{m}} dx 			
\\
					&= c_{-} \int_{- \infty}^{0} |x|^{\alpha} \phi(x) dx + c_{+} \int_{0}^{\infty} x^{\alpha} \phi(x) dx ,
			\end{align*}
with $c_{-}$ and $c_{+}$ given by \eqref{eq:strucalpha>-1constants}.
		
			\end{proof}
	
\subsection{Structural Theorem for negative integral degrees}
\label{subsection Structural theorem for negative integral degree}
We now address the case of quasiasymptotics of degree $\alpha\in\mathbb{Z}_{-}$. The next structural theorem is the second main result of this section.

\begin{theorem}
		\label{t:fstructnid}
	Let $k \in \integers_{+}$ and $f \in \uddspacereal{\beurou}$. Then, $f$ has the quasiasymptotic behavior 
\begin{equation}
\label{eq:strucquasi-k} f(\lambda x)\sim \frac{L(\lambda)}{\lambda^{k}}( \gamma \delta^{(k-1)}(x) + \beta x^{-k}) \quad \mbox{ in } \uddspacereal{\beurou} \mbox{ as }\lambda\to\infty 
\end{equation}	
if and only if there exist continuous functions $f_{m}$ on $\reals$, $m \geq k - 1$, such that
			\begin{equation}
				\label{eqext:fstructnid}
				f = \sum_{m = k - 1}^{\infty} f_{m}^{(m)} , 
			\end{equation}
the limits	
			\begin{equation}
				\label{eq:fstructnidasymp}
				\lim_{x \rightarrow \pm \infty} \frac{f_{m}(x)}{x^{ m-k} L(|x|)}=c_{m}^{\pm}, \qquad m \geq k - 1, 
			\end{equation}
and			
		\begin{equation}
				\label{eq:fstructnidasymp2}
				\lim_{x \rightarrow\infty} \frac{1}{ L(x)}\int_{-x}^{x}f_{k-1}(t)dt=c_{k-1}^{\ast}
			\end{equation}	
		exist, and for some $\ell > 0$ (any $\ell > 0$) there is $C=C_{\ell} > 0$ such that
			\begin{equation}
				\label{eq:fstructnidbound}
				\left| f_{m}(x) \right| \leq  C \frac{\ell^{m}}{M_{m}}(1+ |x|)^{m-k} L(|x|), \qquad x\in\mathbb{R},
			\end{equation}
		for all $m \geq k $. Furthermore, we must have
		\begin{equation}
\label{eq:constantsk}
\gamma= c_{k-1}^{\ast}+\sum_{m=k}^{\infty}(c^{+}_{m}-c^{-}_{m}) \quad \mbox{and} \quad \beta= (-1)^{k-1}(k-1)!c^{+}_{k-1}=(-1)^{k-1}(k-1)!c^{-}_{k-1}.
\end{equation}		
		
	\end{theorem}
	
	\begin{proof} In view of Lemma \ref{l:qaprimitiveinfinity}(ii) we may assume that $k=1$.
	
	\emph{Necessity.} We start showing the necessity of the conditions if $f$ has the quasiasymptotic behavior \eqref{eq:strucquasi-k}. 
Take a compactly supported ultradistribution $f_c$ that coincides with $f$ on $[-e,e]$ and consider $\tilde{f}=f-f_{c}$, so that $\supp (f-f_{c})\cap [-e,e]=\emptyset$. 
We set  $g(x) = x(f(x)-f_{c}(x))$, which
  has quasiasymptotic behavior 
	$g(\lambda x)\sim \beta L(\lambda)$ in   $\uddspacereal{\beurou}$ as $\lambda\to\infty. $
Splitting $g$ as the sum of two distributions supported on $(-\infty,-e)$ and $(e,\infty)$ respectively, we can apply Lemma \ref{l:quasiStrucSup} to obtain its structure as
			$ g = \sum_{m = 0}^{\infty} g_{m}^{(m)} , $
where each of the functions has support in $(-\infty, -e)\cup (e, \infty)$, satisfies the corresponding bounds implied by the lemma, and is such that the limits $\lim_{x\to \pm \infty}x^{-m}g_{m}(x)/L(|x|)$ exist.  Define, for any $j \in \naturals$, the following continuous functions
			\[ \tilde{f}_{j}(x) = \frac{x^{j - 1}}{j!} \sum_{m = j}^{\infty} m! g_{m}(x) x^{-m} , \qquad x \neq 0 , \]
and $\tilde{f}_{j}(0)=0$.
		Let us verify they satisfy the requirements that the $f_j$ should satisfy. First of all, for some $\ell > 0$ (any $\ell > 0$) and $C=C_{\ell}>0$,
			\[ \left| \tilde{f}_{j}(x) \right| \leq C \frac{|x|^{j - 1}}{j!} L(|x|) \sum_{m = j}^{\infty} \frac{m! \ell^{m}}{M_{m}} \leq C' |x|^{j - 1} \frac{\ell^{j}}{M_{j}} L(|x|) , \]
		by \eqref{eq:Msumfacgrowth} from Lemma \ref{l:Msumfacgrowth}. This not only shows  that each $\tilde{f}_{j}$ is well-defined and continuous on $\reals$, but also provides the bounds \eqref{eq:fstructnidbound} for them. From dominated convergence we infer the existence of
			\[ \lim_{x \rightarrow \pm \infty} \frac{\tilde{f}_{j}(x)}{x^{j - 1} L(|x|)} = \lim_{x \rightarrow \pm \infty} \frac{1}{j!} \sum_{m = j}^{\infty} \frac{m! g_{m}(x)}{x^{m} L(|x|)} = \frac{1}{j!} \sum_{m = j}^{\infty} \lim_{x \rightarrow \pm\infty} \frac{m! g_{m}(x)}{x^{m} L(|x|)} .\]
Take an arbitrary $\phi \in \udspacereal{\beurou}$ and let $\varphi \in \udspacereal{\beurou}$ be another corresponding test function that coincides with $\phi$ on $\mathbb{R}\setminus{(-e,e)}$, while its support does not contain the origin. We then have
		 	\begin{align*} 
				\ev{\tilde{f} (x)}{\phi(x)} &=  \ev{g(x)}{\frac{\varphi(x)}{x}} 
			= \sum_{m = 0}^{\infty} \sum_{j = 0}^{m} (-1)^{m} {m \choose j} \ev{g_{m}(x)}{(-1)^{m - j} (m - j)! \frac{\varphi^{(j)}(x)}{x^{m - j + 1}}} \\
										  &= \sum_{j = 0}^{\infty} \frac{(-1)^{j}}{j!} \sum_{m = j}^{\infty} \ev{m! x^{j- 1} \frac{g_{m}(x)}{x^{m}}}{\varphi^{(j)}(x)} 
										  = \sum_{j = 0}^{\infty} \ev{\tilde{f}_{j}^{(j)}}{\phi(x)} . 
			\end{align*}
Applying the first structural theorem to $f_c$ as in the proof of Theorem \ref{t:structalpha>-1}, we obtain compactly supported continuous functions $g_{m}$ such that $f_m=\tilde{f}_m+g_m$ satisfy \eqref{eqext:fstructnid}, \eqref{eq:fstructnidasymp}, and \eqref{eq:fstructnidbound}.
The necessity of \eqref{eq:fstructnidasymp2} follows from \eqref{eq:asymptotically homogeneous in quasi} below. That \eqref{eq:constantsk} must necessarily hold will also be shown below in the proof of the converse.	
	
\emph{Sufficiency.} Conversely, assume that \eqref{eqext:fstructnid} holds with $f_m$ fulfilling \eqref{eq:fstructnidasymp}, \eqref{eq:fstructnidasymp2} and \eqref{eq:fstructnidbound} (recall we work with the reduction $k=1$). We assume without loss of generality that $L(x)$ is everywhere continuous and vanishes for $x\leq 1$. We consider
$
g= \sum_{m = 1}^{\infty} f_{m}^{(m - 1)}.
$
It follows from Theorem \ref{t:structalpha>-1} that $g$ has quasiasymptotic behavior of degree 0 with respect to $L(\lambda)$, and differentiation then yields
$$
f(\lambda x)-f_{0}(\lambda x)= g'(\lambda x)\sim (\gamma-c_{0}^{*})  \frac{L(\lambda)}{\lambda}\delta(x) \quad \mbox{ in } \uddspacereal{\beurou} \mbox{ as }\lambda\to\infty,
$$
with $\gamma$ precisely given as in (\ref{eq:constantsk}). It thus remains to determine the quasiasymptotic properties of $f_{0}$. Write $F(x)= \int_{0}^{x} f_{0}(t)dt$. Since $f_{0}(\pm x)\sim \pm c^{\pm}_{0}L(x)/x$, $x\to\infty,$ one readily shows that
\begin{align*}
F(\lambda x)H(\pm x)&= F(\pm \lambda)H(\pm x)+ c_{0}^{\pm} \int_{\lambda}^{\pm \lambda x}\frac{L(t)}{t}dt+ o\left(L(\lambda)\right) 
\\
&= F(\pm \lambda)H(\pm x)+ c^{\pm}_{0}L(\lambda)H(\pm x)\log |x| + o\left(L(\lambda)\right), \quad \lambda \to\infty,
\end{align*}
uniformly for $x$ on compact intervals, and in particular the relation holds in $\uddspacereal{\beurou}$. Differentiating
\[
F(\lambda x)= F(-\lambda)H(-x)+ F(\lambda) H(x)+ L(\lambda)\left(c^{-}_{0}H(-x)+ c^{+}_{0} H(x)\right)\log |x|+o
\left(L(\lambda)\right) ,
\]
we conclude that 
\begin{equation}
\label{eq:asymptotically homogeneous in quasi}
f_{0}(\lambda x)= \frac{F(\lambda)-F(-\lambda)}{\lambda} \delta(x)+ \frac{L(\lambda)}{\lambda}\left(c_{0}^{-}\operatorname*{Pf}\left(\frac{H(-x)}{x}\right)+ c_{0}^{+}\operatorname*{Pf}\left(\frac{H(x)}{x}\right)\right)+ o\left(\frac{L(\lambda)}{\lambda}\right),
\end{equation}
whence the result follows.		
	\end{proof}

\subsection{Extension from $\reals\setminus\{0\}$ to $\reals$}
\label{subsec:extension}
The methods employed in the previous two subsections also allow us to study the following question. Suppose that the restriction of $f\in \uddspacereal{\beurou}$
 to $\mathbb{R}\setminus\{0\}$ is known to have quasiasymptotic behavior in $\uddspace{\beurou}{\mathbb{R}\setminus\{0\}}$, what can we say about the quasiasymptotic properties of $f$? In view of symmetry considerations, it is clear that it suffices to restrict our attention to ultradistributions supported on $[0,\infty)$. 
	\begin{theorem}
		\label{t:extensioninftyalpha}
Suppose that $f \in \uddspacereal{\beurou}$ is supported in $[0,\infty)$ and has quasiasymptotic behavior 
$
f(\lambda x)\sim c \lambda^{\alpha} L(\lambda)x^{\alpha}$ in  $\uddspacereal{\beurou} (0,\infty)$ as  $\lambda\to\infty.$

			\begin{itemize}
				\item[(i)] If $\alpha > -1$, then $f(\lambda x)\sim c \lambda^{\alpha} L(\lambda)x_{+}^{\alpha}$ in $\uddspacereal{\beurou}$ as $\lambda\to\infty$.
				\item [(ii)] If $\alpha < -1$ and $N \in \naturals$ is such that $-(N + 1) < \alpha < -N$, then there exist constants $a_{0}, \ldots, a_{N - 1}$ such that
					\[ f(\lambda x) - \sum_{n = 0}^{N - 1} a_{n} \frac{\delta^{(n)}(x)}{\lambda^{n+1}} \sim c \lambda^{\alpha} L(\lambda)x_{+}^{\alpha} \quad \mbox{ in } \uddspacereal{\beurou}  \mbox{ as } \lambda\to\infty, \]
	
				\item [(iii)] If $\alpha=-k\in\mathbb{Z}_{-}$, then there is a function $b$ satisfying\footnote{Such functions are called associate homogeneous of degree 0 with respect to $L$ in \cite{stevan2011asymptotic,StructThQADistrInf}. They coincide with functions of the so-called De Haan class \cite{bingham1989regular}.} for each $a>0$
			\begin{equation}
				\label{eq:asshomogeneous}
				b(ax) = b(x) + c\frac{(-1)^{k - 1}}{(k - 1)!}  L(x) \log a + \littleoh{L(x)},
			\end{equation}
$x \rightarrow \infty$,	and constants $a_{0}, \ldots, a_{k - 1}$ such that
			\begin{equation}
				\label{eq:extensioninftynid}
			 	f(\lambda x) = c \frac{L(\lambda)}{\lambda^{k}} \operatorname*{Pf}\left(\frac{H(x)}{x^{k}}\right) + \frac{b(\lambda)}{\lambda^{k}} \delta^{(k - 1)}(x) + \sum_{j = 0}^{k - 1} a_{j} \frac{\delta^{(j)}(x)}{\lambda^{j + 1}} + \littleoh{\frac{L(\lambda)}{\lambda^{k}}} ,
			\end{equation}
		 in $\uddspacereal{\beurou}$ as $\lambda \rightarrow \infty$. 
			\end{itemize}
	\end{theorem} 
	
	\begin{proof}
	The moment asymptotic expansion \cite{Schmidt2005} says that we may assume that, say, $\supp f\subset (e,\infty)$ by removing a neighborhood of the origin. So, we can apply exactly the same argument as in the proof of Theorem \ref{t:structalpha>-1} (via Lemma \ref{l:quasiStrucSup} and Lemma \ref{l:qaprimitiveinfinity}(i)) to show parts (i) and (ii). For (iii), we assume without loss of generality that $k=1$ (Lemma \ref{l:qaprimitiveinfinity}(ii)) and apply the same argument as in the proof of Theorem \ref{t:fstructnid} to conclude that 
$f(\lambda x)= f_{0}(\lambda x)+ \gamma L(\lambda)\delta(\lambda x)+o\left(L(\lambda)/\lambda\right)$ in $\uddspacereal{\beurou},$
 where the continuous function $f_0$ has also support in $(e,\infty)$ and satisfies $f_{0}(x)\sim c L(x)/x$, $x\to \infty$, in the ordinary sense. At this point the result can be derived from \cite[Theorem 4.3]{StructThQADistrInf} (see also \cite[Theorem 2.38, p. 155]{stevan2011asymptotic}), but we might argue directly as follows. In fact, we proceed in the same way we arrived at (\ref{eq:asymptotically homogeneous in quasi}). Set $b(x)=\int_{1}^{x}f_{0}(t)dt$,
then, uniformly for $x$ in compact subsets of $(0,\infty)$,
$$
b(\lambda x)=b(\lambda)H(x)+c\int_{\lambda}^{\lambda x}\frac{L(t)}{t}dt+ o(L(\lambda))= b(\lambda)H(x)+cL(\lambda)H(x) \log x+ o(L(\lambda)),
$$
so that differentiation finally shows
$$
f_{0}(\lambda x)= \frac{b(\lambda)}{\lambda} \delta(x)+ c\frac{L(\lambda)}{\lambda}\operatorname*{Pf}\left(\frac{H(x)}{x}\right)+ o\left(\frac{L(\lambda)}{\lambda}\right) \qquad \mbox{in } \mathcal{D}'.
$$	\end{proof}

\section{The structure of quasiasymptotics at the origin}\label{structure quasiasymptotic at the origin}
We now focus our attention on quasiasymptotic behavior at the origin. The reader should notice that Lemma \ref{l:qaprimitiveinfinity} holds for quasiasymptotics at the origin as well. Furthermore, it is a simple consequence of the definition that quasiasymptotics at the origin is a local property, in the sense that two ultradistributions that coincide in a neighborhood of the origin must have precisely the same quasiasymptotic properties. Throughout this section $L$ stands for a slowly varying function at the origin and we set $\tilde{L}(x)=L(1/x)$. From now on, by convention the parameters $\varepsilon\to 0^{+}$ and $\lambda\to \infty$.

We will reduce the analysis of the structure of quasiasymptotics at the origin to that of the quasiasymptotics at infinity the via the change of variables $x\leftrightarrow 1/x$. We therefore need to see how this substitution acts on derivatives.
	\begin{lemma}
		\label{l:faadibrunoreciprocal}
		Let $\phi \in C^{\infty}(\reals\setminus\{0\})$ and set $\psi(x) := x^{-2} \phi(1 / x)$. Then for any $m \in \mathbb{N}$, there exist constants $c_{m, 0}, \ldots, c_{m, m}$ such that
			\begin{equation}
				\label{eq:faadibrunoreciprocal}
				\frac{d^{m}}{dx^{m}} \left( \psi(x) \right) = \sum_{j = 0}^{m}  c_{m, j} \frac{\phi^{(j)}(1 / x)}{x^{m + j + 2}} , 
			\end{equation}
		where we have the bounds
			\begin{equation}
				\label{eq:faadibrunoreciprocalbound}
				|c_{m, j}| \leq \frac{m!}{j!} 4^{m} , \qquad 0 \leq j \leq m . 
			\end{equation}
	\end{lemma}
	
	\begin{proof}
		Applying the Fa\`{a} di Bruno formula \cite[Eq. (2.2)]{Johnson02thecurious},
					\begin{equation*}
				\frac{d^{k}}{dx^{k}} \left( \phi(1/x) \right) 
		=									 \sum_{j = 1}^{k} (-1)^{k} x^{-(k + j)} \phi^{(j)}(1/x) B_{k, j}(1!, 2!, \ldots, (k - j + 1)!) ,
			\end{equation*}
  where $B_{k, j}$ are the Bell polynomials; from their generating function identity \cite[ (3a'), p. 133]{AdvancedCombinatorics}		we infer that 
			\[ B_{k, j}(1!, \ldots, (k - j + 1)!) = \left. \frac{d^{k}}{dt^{k}} \left( \frac{1}{j!} \left( \frac{t}{1 - t} \right)^{j} \right) \right|_{t = 0} = \frac{k!(k - 1)!}{j! (j - 1)!(k-j)!} . \]
Therefore, we obtain that \eqref{eq:faadibrunoreciprocal} holds with 
$$
c_{m,0}= (-1)^{m} (m + 1)! \quad \mbox{ and } \quad c_{m,j}= (-1)^{m}\frac{m!}{j!} \sum_{k=j}^{m} (m-k+1) \binom{k-1}{j-1}
$$
when $0<j\leq m$, whence one readily obtains the bound \eqref{eq:faadibrunoreciprocalbound}.
	\end{proof}

	\begin{theorem}
		\label{t:structoriginalpha>-1}
		Let $\alpha \notin \integers_{-}$ and let $k \in \naturals$ be the smallest integers such that $-(k + 1) < \alpha$. Then, $f \in \uddspacereal{\beurou}$ has quasiasymptotic behavior
		
\begin{equation}
\label{eqquasiorigin noninteger}		
f(\varepsilon x)\sim \varepsilon^{\alpha}L(\varepsilon)(c_{-}x_{-}^{\alpha}+c
_{+}x_{+}^{\alpha}) \quad \mbox{ in }\uddspacereal{\beurou} \mbox{ as }\varepsilon\to0^{+}
\end{equation}	
		 if and only if there exist functions $f_m\in L^{1}(-1,1)$, $m \geq k$, that are continuous on $[-1,1]\setminus\{0\}$  such that \eqref{eq:fstructalpha>-1} holds on $(-1,1)$,		
			\begin{equation}
				\label{eq:structoriginalpha>-1limit}
				c^{\pm}_{m}=\lim_{x \rightarrow 0\pm} \frac{f_{m}(x)}{x^{m}|x|^{\alpha} L(|x|)}, \qquad m \geq k ,
			\end{equation}
		exist, and furthermore, for some $\ell > 0$ (for any $\ell > 0$) there is a $C > 0$ such that
			\begin{equation}
				\label{eq:structoriginalpha>-1bound}
				|f_{m}(x)| \leq C \frac{\ell^{m}}{M_{m}} |x|^{\alpha + m} L(|x|) , \qquad 0 < |x| \leq 1 , 
			\end{equation}
		for all $m \geq k$. Moreover, the relation 
                \eqref{eq:strucalpha>-1constants} must hold.
	\end{theorem}
	
	\begin{proof}
		The proof of sufficiency can be done analogously as in Theorem \ref{t:structalpha>-1}. Hence we are only left with necessity. If we can show the theorem for degree larger than $-1$, then the full structure theorem will follow from Lemma \ref{l:qaprimitiveinfinity}(i), hence we assume that $\alpha > -1$ (hence $k=0$). If $f$ has quasiasymptotic behavior with respect to $\varepsilon^{\alpha} L(\varepsilon)$, then $\widetilde{f}(x) := f(1 / x)$ has quasiasymptotic behavior in $\uddspace{\beurou}{\reals\setminus\{0\}}$ with respect to $\lambda^{- \alpha} \widetilde{L}(\lambda)$. Then by Theorem \ref{t:structalpha>-1} or Theorem \ref{t:fstructnid} if $\alpha \in \integers_{+}$ and keeping in mind our observations from Section \ref{subsec:extension}, there exist continuous $\widetilde{f}_{m}$ in $\reals\setminus\{0\}$, $m \geq 0$, that satisfy (\ref{eq:fstructalpha>-1}), (\ref{eq:fstructalpha>-1asymp}) and (\ref{eq:fstructalpha>-1bound}). Consider now for any $m \geq 0$,
			\[ f_{m}(x) := \sum_{k = m}^{\infty} (-1)^{k + m } c_{k, m} \widetilde{f}_{k}(1 / x) x^{m + k} , \]
		where the $c_{k, m}$ are as in Lemma \ref{l:faadibrunoreciprocal}. By (\ref{eq:fstructalpha>-1bound}) and (\ref{eq:faadibrunoreciprocalbound}) it follows that for some $\ell > 0$ (for any $\ell > 0$) and any $0 < |x| \leq 1$,
			\begin{align*}
				|f_{m}(x)| &= \left| \sum_{k = m}^{\infty} (-1)^{k + m } c_{k, m} \widetilde{f}_{k}(1 / x) x^{m + k}  \right| 
					       \leq \sum_{k = m}^{\infty} \frac{k!}{m!} 4^{k} \cdot C \frac{(2\ell)^{k}}{M_{k}} |x|^{\alpha - k} L(|x|) |x|^{m + k} \\
					       &= C |x|^{\alpha + m} L(|x|)\frac{1}{m!} \sum_{k = m}^{\infty} k! \frac{(8\ell)^{k}}{M_{k}} \leq C C_{8\ell} \frac{(8\ell)^{m}}{M_{m}} |x|^{\alpha + m} L(|x|) , 
			\end{align*}
		by  \eqref{eq:Msumfacgrowth} from Lemma \ref{l:Msumfacgrowth}. This not only shows existence and continuity in $[-1, 1]\setminus\{0\}$, but also shows that the $f_{m}$ satisfy (\ref{eq:structoriginalpha>-1bound}). By (\ref{eq:fstructalpha>-1asymp}) and dominated convergence, it also follows that for these functions the  limits (\ref{eq:structoriginalpha>-1limit}) exist. Now take any $\phi \in \udspace{\beurou}{\reals\setminus\{0\}}$ with $\supp \phi \subseteq (-1, 1)$ and set $\psi(x) := \phi(1 / x) x^{-2}$. Then,
			\[ \ev{f(x)}{\phi(x)} = \ev{\widetilde{f}(x)}{\psi(x)} = \sum_{k = 0}^{\infty} \ev{\widetilde{f}_{k}(x)}{(-1)^{k} \psi^{(k)}(x)} . \]
		Since for any $k \in \naturals$, by Lemma \ref{l:faadibrunoreciprocal},
			\begin{align*}
				\int_{- \infty}^{\infty} \widetilde{f}_{k}(x) \psi^{(k)}(x) dx 
								= \sum_{m = 0}^{k} c_{k, m} \int_{- \infty}^{\infty}  \widetilde{f}_{k}(1 / x) \phi^{(m)}(x) x^{m + k} dx ,
			\end{align*}
		it follows by switching the order of summation that
			$f = \sum_{m = 0}^{\infty} f_{m}^{(m)} ,$
		in $\uddspacereal{\beurou}((-1, 1)\setminus\{0\})$. Now as $\alpha > -1$, the latter sum is an element of $\uddspacereal{\beurou}$, so that there is some $g \in \uddspacereal{\beurou}$ with $\supp g \subseteq \{0\}$ for which
			$ f = \sum_{m = 0}^{\infty} f_{m}^{(m)} + g , $
		in $\uddspacereal{\beurou}(-1, 1)$. Since we have already shown sufficiency, the sum has quasiasymptotics with respect to $\varepsilon^{\alpha} L(\varepsilon)$, implying that the same holds for $g$. If $g\neq 0$, we can find an ultradifferentiable operator $P(D) = \sum_{n \geq n_0} a_{n} D^{n}$ of type $\beurou$ such that $g = P(D) \delta$ and $a_{n_0}\neq 0$. Then, for any $\phi \in \udspacereal{\beurou}$,
			\[ \ev{\frac{g(\varepsilon x)}{\varepsilon^{\alpha} L(\varepsilon)}}{\phi(x)} = \sum_{n = n_0}^{\infty} (-1)^{n} a_{n} \frac{\varepsilon^{- n - \alpha - 1}}{L(\varepsilon)} \phi^{(n)}(0) . \]
But if $\phi(x)=x^{n_0}$ in a neighborhood of 0, we conclude that 

$$
\infty=\lim_{\varepsilon \to 0^{+}}\frac{1}{\varepsilon^{n_0 +\alpha+  1} L(\varepsilon)}= \frac{(-1)^{n_0}}{a_{n_0}n_0!} \lim _{\varepsilon\to 0^{+}}\ev{\frac{g(\varepsilon x)}{\varepsilon^{\alpha} L(\varepsilon)}}{\phi(x)},
$$
leading to a contradiction. Therefore, $g$ must be identically 0 and this completes the proof of the theorem. 	\end{proof}
	
The structure for negative integral degree can be described as follows. 

	\begin{theorem}
		\label{t:structoriginnid} Let $f \in \uddspacereal{\beurou}$  and $k \in \integers_{+}$. Then, $f$ has quasiasymptotic behavior 
		
		\begin{equation}
\label{eq:strucquasi-k-0} f(\varepsilon x)\sim \frac{L(\varepsilon)}{\varepsilon^{k}}( \gamma \delta^{(k-1)}(x) + \beta x^{-k}) \quad \mbox{ in } \uddspacereal{\beurou} \mbox{ as }\varepsilon\to0^{+} 
\end{equation}	
 if and only if there are continuous functions $F$ and $f_{m}$ on $[-1, 1]\setminus\{0\}$, $m \geq k$, such that
			\begin{equation}
				\label{eq:structoriginnid}
				f = F^{(k)}+\sum_{m = k }^{\infty} f_{m}^{(m)} \qquad \text{ on } (-1, 1) , 
			\end{equation}
the limits		
			\begin{equation}
				\label{eq:structoriginnidasymp}
				c^{\pm}_{m}=\lim_{x \rightarrow 0\pm} \frac{f_{m}(x)}{x^{ m-k} L(|x|)} , \qquad m \geq k ,
			\end{equation}
			exist,
for some $\ell > 0$ (for any $\ell > 0$) there exists $C =C_{\ell}> 0$ such that 
			\begin{equation}
				\label{eq:structoriginnidbound}
				|f_{m}(x)| \leq C \frac{\ell^{m}}{M_{m}} |x|^{ m-k} L(|x|) , \qquad 0 < |x| \leq 1 ,
			\end{equation}
		for all $m \geq k$, and for any $a > 0$ the limit
	\begin{equation}
				\label{eq:structoriginnidprimitive}
				\lim_{x \rightarrow 0+} \frac{F(ax) -F(-x)}{L(x)} = c^{\ast}_{1} + c^{\ast}_{2} \log a
			\end{equation}		
		exists. In this case,
		
		\begin{equation}
\label{eq:constantskzero}
\gamma= c^{\ast}_1+\sum_{m=k}^{\infty}(c^{+}_{m}-c^{-}_{m}) \quad \mbox{and} \quad \beta= (-1)^{k-1}(k-1)!c^{\ast}_{2}.
\end{equation}	 
	\end{theorem}
	
	\begin{proof}	
For the sufficiency, applying Theorem \ref{t:structoriginalpha>-1} to the series $\sum_{m=k}^{\infty}f^{(m-1)}_{m}$, one deduces
$
f(\varepsilon x)- F^{(k)}(\varepsilon x)\sim (\gamma- c^{\ast}_{1}) \delta^{(k-1)}(\varepsilon x)$  in  $\uddspacereal{\beurou}$ as $\varepsilon\to0^{+}  
$.
In view of \cite[Theorem 5.3]{StructThQADistrOrigin} (see also \cite[Theorem 2.33, p. 149]{stevan2011asymptotic}), we have
$
F^{(k)}(\varepsilon x)\sim L(\varepsilon)( c^{\ast}_{1} \delta^{(k-1)}(\varepsilon x) + \beta (\varepsilon x)^{-k})$ in  $\mathcal{D}'$ as $\varepsilon\to0^{+},$
which yields the result.

For the necessity, we may assume that $k=1$. We now apply Theorem \ref{t:structoriginalpha>-1} to $xf(x)$. Using the same reasoning as in the proof of Theorem \ref{t:fstructnid}, one can write $f(x)=f_{0}+\sum_{m=1}^{\infty}f_{m}^{(m)}$ on $(-1,1)\setminus\{0\}$, with continuous functions $f_0, f_1,\dots$ on $[-1,1]\setminus\{0\}$ such that the limits \eqref{eq:structoriginnidasymp} exist including the case $m=0$. Applying again Theorem \ref{t:structoriginalpha>-1} to the series $\sum_{m=1}^{\infty}f_{m}^{(m-1)}$, we deduce that $f_{0}$ has an extension $g_0$ to $\mathbb{R}$ with quasiasymptotic behavior of order $-1$ with respect to $L(\varepsilon)$. Let $F$ be a first order primitive of $g_0$. Due to the fact that $F'=f_0$ off the origin and the quasiasymptotic behavior of $F'$, it is clear that $F$ is integrable at the origin and that it must have the form
\[F(x)=
- H(x) \left(\int^{1}_{x} f_{0}(t)dt + C_{+}\right) + H(-x)\left(\int_{-1}^{x} f_{0}(t)dt + C_{-}\right) . 
\]
Similarly as in the proof of Theorem \ref{t:fstructnid}, we conclude that 
$$
c_1^{\ast}= \lim_{x\to 0^{+}} \frac{F(x)-F(-x)}{L(x)}
$$
must exist
by comparing with the quasiasymptotics of $g_{0}$. Hence, for each $a>0$
\begin{equation*}
\lim_{x \rightarrow 0+} \frac{F(ax) -F(-x)}{L(x)} = c_{1}^{\ast}+ \lim_{x \rightarrow 0+} \frac{1}{L(x)}\int_{x}^{ax}f_{0}(t)dt
=c_{1}^{\ast}+c_{0}^{+}\log a.
\end{equation*}
	\end{proof}
Our method also yields:

\begin{theorem}
		\label{t:extensionzeroalpha}
Suppose that $f_{0} \in \uddspacereal{\beurou}(0,\infty)$ has quasiasymptotic behavior 
\[
f_0(\varepsilon x)\sim c \varepsilon^{\alpha} L(\varepsilon)x^{\alpha} \quad \mbox{in } \uddspacereal{\beurou} (0,\infty) \mbox{ as } \varepsilon\to0^{+}.
\]
Then $f_0$ admits extensions to $\mathbb{R}$. Let $f \in \uddspacereal{\beurou}$ be any of such extensions with support in $[0,\infty)$. Then:
			\begin{itemize}
				
				\item [(I)] If $\alpha \notin \mathbb{Z}_{-}$, then there is $g\in  \uddspacereal{\beurou}$ with $\operatorname*{supp} g\subseteq \{0\}$ such that 

					\[ f(  \varepsilon x) - g(\varepsilon x) \sim c  \varepsilon^{\alpha} L( \varepsilon)x_{+}^{\alpha} \quad \mbox{ in } \uddspacereal{\beurou}  \mbox{ as }  \varepsilon\to0^{+}. \]
	
				\item [(II)] If $\alpha=-k\in\mathbb{Z}_{-}$, then there are a function $b$ satisfying \eqref{eq:asshomogeneous} as $x\to0^{+}$ for each $a>0$
		and an ultradistribution $g\in  \uddspacereal{\beurou}$ with $\operatorname*{supp} g\subseteq \{0\}$ such that 
			\begin{equation*}
			 	f( \varepsilon x) = c \frac{L( \varepsilon)}{ \varepsilon^{k}} \operatorname*{Pf}\left(\frac{H(x)}{x^{k}}\right) + \frac{b( \varepsilon)}{ \varepsilon^{k}} \delta^{(k - 1)}(x) + g(\varepsilon x) + \littleoh{\frac{L( \varepsilon)}{ \varepsilon^{k}}} \quad\mbox{in } \uddspacereal{\beurou} \:\mbox{as }  \varepsilon \rightarrow 0^{+}.
			\end{equation*}
			\end{itemize}
	\end{theorem}

\section{Quasiasymptotic behavior in $\mathcal{Z}^{\prime\beurou}$}
\label{quasiasymptotics Gelfand-Shilov}
As an application of our structural theorems, we now discuss some other extension results for quasiasymptotics of ultradistributions. For distributions, the connection between tempered distributions and the quasiasymptotic behavior has been extensively studied \cite{stevan2011asymptotic, P-V2014,StructThQADistrInf,StructThQADistrOrigin,Zavyalov89}. The following properties are well known:

\begin{enumerate}
\item[1.] If $f\in\mathcal{D}'$ has quasiasymptotic behavior at infinity, then $f\in \mathcal{S}'$ and it has the same quasiasymptotic behavior in $\mathcal{S}'$.
\item [2.]If $f\in\mathcal{S}'$ has quasiasymptotic behavior at the origin in $\mathcal{D}'$, then it has the same quasiasymptotic behavior in $\mathcal{S}'$.
\end{enumerate}

Our goal is to obtain ultradistributional analogs of these results. For this, we introduce new ultradistibution spaces $\mathcal{Z}^{\prime\beurou}$, somewhat reminiscent of the Gelfand-Shilov type spaces \cite{carmichael2007boundary} and at the same time generalizing $\mathcal{S}'$. They are defined as follows. For any $n \in \naturals$ and $h > 0$, $\mathcal{Z}^{M_{p}, h}_{n}$ denotes the Banach space of all $\varphi \in C^{\infty}$ for which the norm
	\[ \norm{\varphi}_{M_{p}, n, h} := \sup_{x \in \reals, m \in \naturals} \frac{(1 + |x|)^{n + m} |\varphi^{(m)}(x)|}{h^{m} M_{m}} \]
is finite. Then we consider the following locally convex spaces
	\[ \mathcal{Z}^{(M_{p})}_{n} = \varprojlim_{h \rightarrow 0^{+}} \mathcal{Z}^{M_{p}, h}_{n}, \qquad \mathcal{Z}^{\{M_{p}\}}_{n} = \varinjlim_{h \rightarrow \infty} \mathcal{Z}^{M_{p}, h}_{n} , \]
corresponding to the Beurling and Roumieu case, where we use $\mathcal{Z}^{\beurou}_{n}$ as a common notation for these two cases, and finally we define
	\[ \mathcal{Z}^{\beurou} = \varprojlim_{n \in \naturals} \mathcal{Z}^{\beurou}_{n} . \]
The aim of this section is to show that quasiasymptotic behavior in $\mathcal{D}^{\prime\beurou}$ naturally extends to quasiasymptotic behavior in $\mathcal{Z}^{\prime\beurou}$. Let us first consider the case at infinity.

\begin{theorem}
	If $f \in \mathcal{D}^{\prime\beurou}$ has quasiasymptotic behavior with respect to $\lambda^{\alpha} L(\lambda)$, with $L$ slowly varying at infinity and $\alpha \in \reals$, then $f \in \mathcal{Z}^{\prime\beurou}$ and it has the same quasiasymptotic behavior in $\mathcal{Z}^{\prime\beurou}$.
\end{theorem}

\begin{proof}
Let $k \in \naturals$ be the smallest natural number such that $-(k + 1) \leq \alpha$. Then by either Theorem \ref{t:structalpha>-1} or Theorem \ref{t:fstructnid} we find some $\ell > 0$ (for any $\ell > 0$) and a $C=C_{\ell} > 0$ such that \eqref{eq:fstructalpha>-1}
	and 	\eqref{eq:fstructalpha>-1bound} hold. Wet set $n=\lceil \alpha + 1 \rceil$. Employing Potter's estimate \eqref{eq:slowlyvaryingsubpolynomialbound} (with $\gamma=\lambda=1$), we find that for any $\varphi \in \mathcal{Z}^{\beurou}$ and any $m \geq k$ we have
			\[ \left| \int_{- \infty}^{\infty} f_{m}(x) \varphi^{(m)}(x) dx \right| \leq C \frac{\ell^{m}}{M_{m}} \int_{\reals} (1 + |x|)^{ m +n} | \varphi^{(m)}(x) | dx \leq C' \|\varphi\|_{M_p,n + 2,h} (h \ell^{m}) , \]
		and as $h \ell$ may be chosen freely, it follows that this is absolutely summable over $m \geq k$. Consequently, $f=\sum_{m=k}^{\infty} f_m^{(m)} \in \mathcal{Z}^{\prime\beurou}$.

		For the quasiasymptotic behavior of $f$, the case where $\alpha$ is not a negative integer can be shown in a similar fashion as the sufficiency proof of Theorem \ref{t:structalpha>-1}. For $\alpha = - k \in \integers_{-}$, it is clear that we only need to treat the case $k = 1$, as the general case then automatically follows by differentiating. By Theorem \ref{t:fstructnid}, there exist continuous functions $f_{m}$, $m \in \naturals$, satisfying (\ref{eq:fstructnidasymp}),  (\ref{eq:fstructnidasymp2}), and (\ref{eq:fstructnidbound}) such that
			$f = f_{0} + \sum_{m = 1}^{\infty} f^{(m)}_{m} . $
		The infinite sum in the previous identity clearly has a primitive with quasiasymptotic behavior with respect to $L(\lambda)$, so that its quasiasymptotic behavior may be extended to the whole of $\mathcal{Z}^{\prime\beurou}$, and in turn its derivative  $\sum_{m = 1}^{\infty} f^{(m)}_{m}$ has quasiasymptotic behavior with respect to $\lambda^{-1} L(\lambda)$ in $\mathcal{Z}^{\prime\beurou}$. By (\ref{eq:fstructnidasymp}) and \eqref{eq:fstructnidasymp2}, $f_{0}$ has quasiasymptotic behavior with respect $\lambda^{-1} L(\lambda)$  in ${\mathcal{D}}'$, hence, by \cite[Remark 3.1]{StructThQADistrInf} (see also \cite[Theorem 2.41, p. 158]{stevan2011asymptotic}), it has the same quasiasymptotic behavior in ${\sss}'$, hence certainly also in $\mathcal{Z}^{\prime\beurou}$. Therefore, the same also holds for $f$.
\end{proof}

Let us now turn our attention to the case at the origin. The next lemma proves that the quasiasymptotic at the origin in $\mathcal{Z}^{\prime\ast}$ is a local property.
\begin{lemma}
	\label{p:quasiasymporiginlocalZ}
	Let $L$ be a slowly varying function at the origin and $\alpha \in \reals$. Suppose $f_{1}, f_{2} \in \mathcal{Z}^{\prime\beurou}$ are such that for some $a > 0$, $f_{1}$ and $f_{2}$ coincide on $\reals \setminus [-a, a]$. Suppose that $f_{1}(\varepsilon x) \sim \varepsilon^{\alpha} L(\varepsilon) g(x)$ in $\mathcal{Z}^{\prime\beurou}$ as $\varepsilon \rightarrow 0^{+}$, then, also $f_{2}(\varepsilon x) \sim \varepsilon^{\alpha} L(\varepsilon) g(x)$ in $\mathcal{Z}^{\prime\beurou}$.
\end{lemma}

\begin{proof}
	We only show the Beurling case; the Roumieu case 	can be shown analogously by employing a projective description for $\mathcal{Z}^{\{M_{p}\}}$ obtained similarly as in  \cite{carmichael2007boundary}. It suffices to show that if $f \in \mathcal{Z}^{\prime(M_p)}$ vanishes near the origin, then $f(\varepsilon x) \sim \varepsilon^{N} \cdot 0$ for all $N \in \naturals$. Let $f$ be as described, then there exist $0 < R < 1$, $n \in \naturals$, $\ell, C > 0$ such that
		\[ \left| \ev{f(\varepsilon x)}{\phi(x)} \right| \leq C \sup_{|x| \geq R, m \in \naturals} \frac{|x|^{n + m} |\phi^{(m)}(x)|}{\ell^{m} M_{m}}, \qquad \phi \in \mathcal{Z}^{(M_{p})} . \]
	Taking $\phi(x) = \varepsilon^{-1} \varphi(x / \varepsilon)$ with $\varphi \in \mathcal{Z}^{(M_{p})}$ and arbitrary $0 < \varepsilon < 1$ we have for $N \geq n$
		\begin{align*}
			\varepsilon^{-N} \left|\ev{f(\varepsilon x)}{\varphi}\right| &\leq C \sup_{|x| \geq R, m \in \naturals} \frac{|x|^{n + m} |\varphi^{(m)}(x / \varepsilon)|}{\varepsilon^{N + m + 1} \ell^{m} M_{m}} \\
			&\leq CR^{-N+n-1} \sup_{|x| \geq R / \varepsilon, m \in \naturals} \frac{|x|^{N + m + 1} |\varphi^{(m)}(x)|}{\ell^{m} M_{m}} \rightarrow 0 , 
		\end{align*}
	as $\varepsilon \rightarrow 0^{+}$.
\end{proof}

\begin{theorem}
	Suppose $f \in \mathcal{Z}^{\prime\beurou}$ has quasiasymptotic behavior in $\mathcal{D}^{\prime\beurou}$ with respect to $\varepsilon^{\alpha} L(\varepsilon)$, with $L$ slowly varying at the origin and $\alpha \in \reals$, then $f$ has the same quasiasymptotic behavior in $\mathcal{Z}^{\prime\beurou}$. 
\end{theorem}

\begin{proof}
	By Lemma \ref{p:quasiasymporiginlocalZ} we may assume that $\supp f \subset [-1, 1]$. Suppose first that $\alpha \notin \integers_{-}$ and let $k \in \naturals$ be the smallest integer such that $-(k + 1) < \alpha$. From Theorem \ref{t:structoriginalpha>-1} we find continuous functions $f_{m}$ on $[-1, 1] \setminus \{0\}$, satisfying (\ref{eq:fstructalpha>-1}), (\ref{eq:structoriginalpha>-1limit}) and (\ref{eq:structoriginalpha>-1bound}). Take any $\psi \in \mathcal{Z}^{\prime\beurou}$ and decompose it as $\psi = \psi_{-} + \psi_{c} + \psi_{+}$ where $\supp \psi_{-} \subseteq (-\infty, -1]$, $\psi_{c}$ has compact support and $\supp \psi_{+} \subseteq [1, \infty)$. Then by the hypothesis
			\[ \lim_{\varepsilon \rightarrow 0+} \ev{\frac{f(\varepsilon x)}{\varepsilon^{\alpha} L(\varepsilon)}}{\psi_{c}(x)} = c_{+} \ev{x_{+}^{\alpha}}{\psi_{c}(x)} + c_{-} \ev{x_{-}^{\alpha}}{\psi_{c}(x)} . \]
		It suffices to show that the same limit holds for $\psi_{\pm}$ placed instead of $\psi_{c}$. As the two cases are symmetrical, we only look at $\psi_{+}$. It follows from (\ref{eq:slowlyvaryingsubpolynomialbound}), (\ref{eq:structoriginalpha>-1bound}) and the Lebesgue dominated convergence theorem that for any $m \geq k$,
			\begin{align*}
				\lim_{\varepsilon \rightarrow 0+} \ev{\frac{f_{m}^{(m)}(\varepsilon x)}{\varepsilon^{\alpha} L(\varepsilon)}}{\psi_{+}(x)}
				&
				= \lim_{\varepsilon \rightarrow 0+}  \int_{1}^{1 / \varepsilon} \frac{L(\varepsilon x)}{L(\varepsilon)} \left( \frac{f_{m}(\varepsilon x)}{(\varepsilon x)^{\alpha + m} L(\varepsilon x)} \right) x^{\alpha + m} \psi_{+}^{(m)}(x) dx \\
				&= c^{+}_{m} \int_{0}^{\infty} x^{\alpha} \psi_{+}(x) dx .
			\end{align*}
		Then another application of dominated convergence shows that
			\[ \lim_{\varepsilon \rightarrow 0+} \ev{\frac{f(\varepsilon x)}{\varepsilon^{\alpha} L(\varepsilon)}}{\psi_{+}(x)} = c_{+} \ev{x^{\alpha}_{+}}{\psi_{+}(x)} . \]
		This shows the case for $\alpha \notin \integers_{-}$. The case of negative integral degree can then be done as in the proof of \cite[Theorem 6.1]{StructThQADistrOrigin}.
\end{proof}

\subsection*{Acknowledgement}
We thank Andreas Debrouwere for his useful comments that led to improvements in the results of Section \ref{quasiasymptotics Gelfand-Shilov}.


\begin{thebibliography}{99}

\bibitem{bingham1989regular}
{\sc N.~Bingham, C.~Goldie, and J.~Teugels}, {\em Regular variation}, 
  Cambridge University Press, Cambridge, 1989.

\bibitem{carmichael2007boundary}
{\sc R.~Carmichael, A.~Kami{\'n}ski, and S.~Pilipovi{\'c}}, {\em Boundary
  values and convolution in ultradistribution spaces}, 
  World Scientific Publishing Co. Pte. Ltd., Hackensack, NJ, 2007.

\bibitem{AdvancedCombinatorics}
{\sc L.~Comtet}, {\em Advanced combinatorics. The art of finite and infinite expansions}, D. Reidel Publishing Co., Dordrecht, 1974.

\bibitem{debrouwere-vindasCousin} {\sc A.~Debrouwere and J.Vindas}, {\em Solution to the first Cousin problem for vector-valued quasianalytic functions}, Ann. Mat. Pura Appl., 196 (2017), 1983--2003.


\bibitem{drozhzhinov2016} {\sc Yu.~N.~Drozhzhinov,}
{\em Multidimensional Tauberian theorems for generalized functions,}
Russian Math. Surveys, 71 (2016), 1081--1134. 

\bibitem{estrada2012distributional}
{\sc R.~Estrada and R.~Kanwal}, {\em A distributional approach to asymptotics.
 Theory and applications}, 
 Birkh{\"a}user Boston, Boston, MA, 2002.

\bibitem{hormander1990analysis}
{\sc L.~H{\"o}rmander}, {\em The analysis of linear partial differential
  operators. I. Distribution theory and Fourier analysis}, 
  Springer-Verlag, Berlin,1990.

\bibitem{Johnson02thecurious}
{\sc W.~P. Johnson}, {\em The curious history of Fa{\`a} di Bruno's
  formula}, Amer. Math. Monthly, 109 (2002), 217--234.

\bibitem{lojasiewicz57} {\sc S.~\L ojasiewicz}, {\em Sur la valeur et la limite d'une distribution en un point}, Studia Math., 16 (1957), 1--36.

\bibitem{ultradistributions1}
{\sc H.~Komatsu}, {\em Ultradistributions. {I}. Structure theorems and a
  characterization}, J. Fac. Sci. Univ. Tokyo Sect. IA Math., 20 (1973), 25--105.



\bibitem{10.2307/44095755}
{\sc S.~Pilipovi{\'c} and B.~Stankovi{\'c}}, {\em Quasi-asymptotics and
  S-asymptotics of ultradistributions},  Bull. Acad. Serbe Sci. Arts Cl. Sci. Math. Natur.,  20 (1995), 61--74.
  

\bibitem{p-s-t} {\sc S.~Pilipovi\'{c}, B.~Stankovi\'{c}, and A.~Taka\v{c}i}, {\em Asymptotic behaviour and Stieltjes transformation of distributions}, Teubner Verlagsgesellschaft, Leipzig, 1990.

\bibitem{stevan2011asymptotic}
{\sc S.~Pilipovi{\'c}, B.~Stankovi{\'c}, and J.~Vindas}, {\em Asymptotic
  behavior of generalized functions}, 
  World Scientific Publishing Co. Pte. Ltd., Hackensack, NJ, 2012.

\bibitem{P-V2014} {\sc S.~Pilipovi\'{c} and J.~Vindas}, {\em Multidimensional Tauberian theorems for vector-valued distributions}, Publ. Inst. Math. (Beograd), 95 (2014), 1--28. 

\bibitem{StirlingNumbersOfSecondKind}
{\sc B.~C. Rennie and A.~J. Dobson}, {\em On Stirling numbers of the second
  kind}, J. Combinatorial Theory, 7 (1969), 116--121.
  
\bibitem{Schmidt2005}  {\sc A.~Schmidt,} {\em Asymptotic hyperfunctions, tempered hyperfunctions, and asymptotic expansions,}
Int. J. Math. Math. Sci. 2005,  755--788. 


\bibitem{StructThQADistrInf}
{\sc J.~Vindas}, {\em Structural theorems for quasiasymptotics of distributions
  at infinity}, Publ. Inst. Math. (Beograd), 84 (2008), 159--174.

\bibitem{VindasReg2011} {\sc J.~Vindas}, {\em Regularizations at the origin of distributions having prescribed asymptotic properties}, Integral Transforms Spec. Funct., 22 (2011), 375--382.

\bibitem{StructThQADistrOrigin}
{\sc J.~Vindas and S.~Pilipovi{\'c}}, {\em Structural theorems for
  quasiasymptotics of distributions at the origin},  Math. Nachr.,
  282 (2009), 1584--1599.
  

\bibitem{vladimirov1988tauberian}
{\sc V.~S.~Vladimirov, Yu.~N.~Drozhzhinov, and B.~I.~Zav'yalov}, {\em Tauberian
  theorems for generalized functions}, Kluwer Academic Publishers, 1988.

\bibitem{Zavyalov}
{\sc B.~I. Zav'yalov}, {\em Self-similar asymptotic form of electromagnetic form-factors and the behavior of their Fourier transforms in the neighborhood of the light cone}, Teoret. Mat. Fiz., 17 (1973), 178--188; English translation in: Theoret. and Math. Phys., 17 (1973), 1074--1081. 

\bibitem{Zavyalov89} {\sc B.~I. Zav'yalov}, {\em Asymptotic properties  of functions that are holomorphic  in
tubular cones}, Math. USSR-Sb., 64
(1989), 97--113.

\bibitem{yang-estrada}
{\sc Y.~Yang and R.~Estrada}, {\em Asymptotic expansion of thick distributions},  Asymptot. Anal., 95 (2015), 1--19.

\end{thebibliography}
\end{document}